\documentclass[12pt]{amsart}

\usepackage{titletoc}

\usepackage[a4paper,twoside,left=3cm,right=3cm,bottom=3.5cm,top=3.5cm]{geometry}

\usepackage{amsmath,amsfonts,amsthm,amssymb,mathrsfs,enumerate,comment}
\usepackage{cite}
\usepackage[all,2cell,dvips]{xy} \SilentMatrices
\usepackage{epsfig,color}
\usepackage{pdftricks}
\usepackage{setspace}

\theoremstyle{plain}
\newtheorem{theorem}{Theorem}[section]

\newtheorem*{lemma*}{Lemma}
\newtheorem{lemma}[theorem]{Lemma}

\newtheorem*{theorem*}{Theorem}

\newtheorem{proposition}[theorem]{Proposition}
\newtheorem*{proposition*}{Proposition}
\newtheorem{cor}[theorem]{Corollary}

\newtheorem*{corollary*}{Corollary}

\theoremstyle{definition}

\newtheorem{remark}[theorem]{Remark}
\newtheorem*{remark*}{Remark}

\newtheorem*{prob}{Problem}

\newtheorem*{definition*}{Definition}
\newtheorem{definition}[theorem]{Definition}

\newtheorem*{example*}{Example}
\newtheorem{example}[theorem]{Example}

\newcommand\sheafExt{\operatorname{\mathcal{E}\mkern-3mu\mathit{xt}}\nolimits}

\numberwithin{equation}{section}

\def\codim{\operatorname{codim}}

\def\c1{\operatorname{c_1}}
\def\c2{\operatorname{c_2}}
\def\Spec{\operatorname{Spec}}
\def\Proj{\operatorname{Proj}}

\def\Sym{\operatorname{Sym}}

\def\CC{{\mathbb C}}
\def\ZZ{{\mathbb Z}}

\def\QQ{{\mathbb Q}}
\def\PP{{\mathbf P}}

\def\L{{  L}}

\def\N{{ N}}
\def\O{{\mathscr O}}
\def\I{{ I}}

\def\E{{\mathscr E}}

\def\H{{\mathscr H}}
\def\F{{\mathscr F}}
\def\G{{\mathscr G}}

\def\M{{\mathscr M}}

\def\+{\oplus}                   
\def\*{\otimes}                  

\def\Pic{\operatorname{Pic}}
\def\Hom{\operatorname{Hom}}
\def\Sym{\operatorname{Sym}}

\def\Bl{\operatorname{Bl}}

\def\cd{\operatorname{cd}}

\renewcommand\setminus{-}

\title{Ample subvarieties and $q$-ample divisors}
\author{John Christian Ottem}
\date{\today}

\address{DPMMS, Cambridge University, Wilberforce Road, Cambridge CB3 0WB, England}
\email{J.C.Ottem@dpmms.cam.ac.uk}

\begin{document}

\thispagestyle{empty}

%
%

\begin{abstract}
We introduce a notion of ampleness for subschemes of any codimension using the theory of $q$-ample line bundles. We also
investigate certain geometric properties satisfied by ample subvarieties, e.g. the Lefschetz hyperplane theorems and numerical positivity. Using these properties, we also construct a counterexample to the converse of the Andreotti-Grauert vanishing theorem.
\end{abstract}
\maketitle
\thispagestyle{empty}

\thispagestyle{empty}

 \section{Introduction}
In addition to being fundamental in many contexts in algebraic geometry, ample divisors have many useful algebro-geometric and topological properties. For example, an ample divisor has an affine complement, which implies that it has non-empty intersection with any other closed subvariety of positive dimension. The goal of this paper is to introduce a notion of ampleness for  subschemes of any codimension and investigate which analogous geometric properties hold for such subschemes. 

Our definition builds on recent work of Demailly, Peternell and Schneider \cite{DPS96} and Totaro \cite{Tot10} by using their notion of a \emph{$q$-ample} line bundle. This is a generalization of the notion of an ample line bundle in the sense that high tensor powers of a line bundle are required to kill cohomology of coherent sheaves in degrees $>q$. If $Y$ is a closed subscheme of a projective variety $X$, we consider the line bundle $\O(E)$ corresponding to the exceptional divisor of the blow-up of $X$ with center $Y$. Even though this line bundle is never ample if $\codim(Y)\ge 2$, we can still expect that ampleness of $Y$ is reflected in some weaker positivity properties of $\O(E)$. In short, we define a codimension $r$ subscheme $Y\subset X$ to be ample \emph{ample} if this line bundle is $(r-1)$-ample.

The first major investigation into ampleness of subvarieties of higher codimension appeared in Hartshorne's book \cite{Har70}. Even though Hartshorne didn't give a definition of ampleness, he surveyed a number of properties such a notion should have, e.g., properties satisfied by complete intersections. In this paper we aim to investigate which of these are satisfied for our definition of ampleness.

The most basic requirement of an ample subscheme is that its
normal bundle should be an ample vector bundle (in the sense of
Hartshorne). For this to make sense we require $Y$ to be a lci subscheme, so that its normal sheaf is a vector bundle. When $Y$ is a Cartier divisor, this condition corresponds to $\O(Y)\big|_Y$ being an ample line bundle on $Y$. Even though this condition is much too weak to be considered ample, it gives some good geometric properties of $Y$ (e.g., non-negative intersection with other varieties). In Section \ref{section:nb}, we show that lci subschemes that are ample in our sense have ample normal bundles.

Another condition concerns the cohomology of the complement $U=X\setminus Y$. When $Y$ is an ample divisor, this complement is affine and so by Serre's affineness criterion $H^i(U,\F)=0\,\,$for all $i>0$ and every coherent sheaf $\F$ on $U$. Generalizing this, we can let $\cd(U)$ denote the smallest positive integer $q$ such that the same condition holds for all $i>q$. In Section \ref{complement}, we show that for an ample subscheme $Y$ in $X$, we have $\cd(X\setminus Y)=\codim Y-1$, in accordance with the case for ample divisors.

The third property is a topological one: ample subvarieties should satisfy some form of the Lefschetz hyperplane theorem. More precisely, if $Y$ is an ample subvariety of $X$, we require that the natural map $H^i(X,\QQ)\to H^i(Y,\QQ)$ should be an isomorphism for $i<\dim Y$ and injective for $i=\dim Y$. In our case, this property follows from the above condition, $\cd(X\setminus Y)=\codim(Y)-1$, and standard topological considerations. Furthermore, in Section \ref{Pn} we will see that these Lefschetz theorems completely characterize ample subvarieties of projective space in the smooth case, so that checking whether a subvariety of projective space is ample amounts to calculating ordinary cohomology groups. 

On the other hand, an interesting feature with our notion of ampleness is that the Lefschetz hyperplane theorem does not necessarily hold with integer coefficients in codimension $\ge 2$. In Section \ref{qpositive} we exploit this observation to construct a counterexample to the converse of the Andreotti-Grauert vanishing theorem, thus giving a negative answer to a problem posed by Demailly, Peternell and Schneider in \cite{DPS96}. In short, this shows that $q$-positivity of a line bundle is a stronger condition than $q$-ampleness for $\frac12{\dim X}-1<q< \dim X-2$.

We also show that ample subvarieties satisfy functorial properties similar to those of ample divisors. For example, ampleness is an open condition in flat families
and is preserved under finite pullbacks. On the other hand, it is worth noting that our definition of ampleness is not preserved
under rational equivalence: A conic in $\PP^3$ is rationally
equivalent to a union of two disjoint lines (which does not satisfy
Lefschetz hyperplane theorem).

Moreover, we show that the ampleness condition has implications for the singularities of $Y$. More precisely, we show that if $Y$ is a reduced ample subscheme and $X$ is smooth, then $Y$ is a locally complete intersection. When $Y$ is not assumed to be reduced, it it is still true that it is equidimensional.

Throughout the paper we work over a field of characteristic 0. The
main reason for this is that it is still unclear what the right notion
of a $q$-ample line bundle should be in positive characteristic. For
example, it is not known whether $q-$ampleness is an open condition in
characteristic $p$ \cite{Tot10}. Moreover, some of the proofs in this
paper require the isomorphism $\Sym^r(\E^*)^*=\Sym^r(\E)$ which is not
valid in this case. Nevertheless, many of the results in this paper
can be modified to characteristic $p$ with minor modifications
(cf. Remark \ref{Gammaample}).

\smallskip

\noindent {\bf Acknowledgements}. I wish to thank my supervisor Burt Totaro for his encouragement and helpful advice. I would also like to thank Claire Voisin and  J\o rgen Vold Rennemo for useful discussions.

\section{Preliminaries on $q$-ample line bundles}\label{prelim}

Let $X$ be an $n$-dimensional projective scheme over a field $k$ of characteristic
0. Following \cite{DPS96} and \cite{Tot10}, we will for a non-negative integer
 $q$, call a line bundle $L$ on $X$ \emph{$q$-ample} if for every coherent sheaf $\F$ there is an integer $m_0>0$ depending on $\F$ such that
\begin{equation}\label{qample:def}
H^i(X,\F\otimes L^{\otimes m})=0
\end{equation}for all $m\ge m_0$ and $i>q$. This is indeed a
generalization of ordinary ampleness in the sense that $0-$ample line
bundles correspond exactly to ordinary ample line bundles by Serre's
theorem. In the following, we will for simplicity not mention the
dependence on $\F$ and write $m\gg 0$ whenever such an $m_0$ exists. 
We will also call a Cartier divisor $D$ on $X$ \emph{$q$-ample} if the corresponding line bundle $\O_X(D)$ is $q$-ample.

\begin{lemma}\label{oneO}
Let $X$ be a projective scheme and fix an ample line bundle $\O(1)$ on $X$. Then $\L$ is $q$-ample if and only if for each $r\ge 0$, $H^i(\L^{\otimes m}\otimes \O(-r))=0$ for $m\gg 0$ and $i>q$. In particular, the condition \eqref{qample:def} needs only be verified for locally free sheaves.
\end{lemma}\begin{proof}
Let $\F$ be a coherent sheaf on $X$. Then $\F(k)$ is globally generated for $k$ sufficiently large, so there is a surjective map $\E\to \F$ where $\E$ is a sum of line bundles of the form $\O_X(-a_i)$. Let $\G$ be the kernel of this map and consider the sequence
$$
0\to \G \to \E\to \F \to 0 
$$ If the condition $H^i(\L^{\otimes m}\otimes \O(-r))=0$ holds for all $i>q$ and $m\gg 0$, we get $H^{i+1}(X, \L^{\otimes m}\otimes \G)=H^{i}(X, \L^{\otimes m}\otimes \F)$ for $i>q$ and $L$ is $q$-ample by downward induction on $q$ (starting with the case $q=n$ where the result is clear).
\end{proof}
In the case $X$ is smooth, or more generally, Gorenstein, Serre duality implies that the $q$-ampleness of $L$ is equivalent to the condition that the dual line bundle $L^*$ kills cohomology in low degrees:

\begin{lemma}\label{smalli}
Let $X$ be a Gorenstein projective scheme of dimension $n$. A line bundle $\L$ on $X$ is $q$-ample if and only if for all locally free sheaves $\E$ and $0\le i < n-q$,
\begin{equation}\label{Hsmalli}
H^i(X,\E \otimes L^{*\otimes m}))=0
\end{equation}for $m\gg 0$. In particular, any non-zero effective divisor is $(n-1)$-ample.
\end{lemma}
\begin{proof}
Let $\E$ be a locally free sheaf on $X$ and suppose $L$ is $q-$ample. By Serre duality, $H^i(X,\E\otimes L^{* \otimes
  m})$ is dual to the cohomology group $H^{n-i}(X,\E^* \otimes \omega_X \otimes L^{ \otimes m}),$ which vanishes for $n-i>q$ and sufficiently large
$m$. Conversely, if the condition \eqref{Hsmalli} is satisfied, we have for any locally free sheaf $\E$,
$H^i(X,\E \otimes L^{\otimes m})=H^{n-i}(X,\E^*\otimes \omega_X\otimes L^{*\otimes m}))^*$. Since $X$ is Gorenstein, $\omega_X$ is locally free and so the last cohomology group vanishes for $i<n-q$ and $m$ large. Hence $L$ is $q$-ample by the previous lemma.

To see why the last statement is true, note that if $D$ is effective, then $\E(-mD)$ cannot have any global sections for $m$ large.
\end{proof}


The following result summarizes some properties of $q$-ample line bundles under base change. Even though the proof is similar to that of the case for $q=0$, we include a proof for convenience of the reader.

\begin{proposition}\label{basechange}
Let $X$ be a projective scheme and let $L$ be a line bundle on $X$. Then:

\begin{enumerate}[i)]
\item $L$ is $q$-ample on $X$ iff $L_{red}$ is $q$-ample on $X_{red}$.
\item $L$ is $q$-ample on $X$ iff $L\big|_{X_i}$ is $q$-ample on each component $X_i$ of $X$.
\item Let $f:X\to Y$ be a finite morphism. Then $L$ $q$-ample implies that $f^*L$ is $q$-ample. If $f$ is surjective, then the converse also holds.
\end{enumerate}
\end{proposition}

\begin{proof}

$i)$ is \cite[Corollary 8.2]{Tot10}. Suppose next that $f$ is a finite morphism. Using the projection formula and the Leray spectral sequence applied to $f$, we find
\begin{equation}\label{leray}
H^i(X,(f^*L)^{\otimes m} \otimes \G)=H^i(Y,L^{\otimes m} \otimes f_*\G),
\end{equation}so if $L$ is $q$-ample, $f^*L$ is also $q$-ample. In particular, the restriction of a $q$-ample line bundle to a subscheme is again $q$-ample. This shows one direction in $ii)$ and the first part of $iii)$.

To prove the remaining direction in $ii)$, let  $X=X_1\cup \cdots \cup X_r$ be the decomposition of $X$ into its connected components and assume $L$ restricted to each $X_i$ is $q$-ample.
Let $I$ be the ideal sheaf of $X_r$ and consider the sequence

$$0 \to I\F \to \F \to \F/I\F \to 0$$

Now, $I\F$ is supported on $X_1\cup \cdots \cup X_{r-1}$ and $\F/I\F$ is supported on $X_n$. Hence by induction on $r$, we have $H^i(X,I\F \otimes L^{\otimes m}) = 0$ and $H^i(X,\F/I\F \otimes L^{\otimes m}) = 0$ for $i>q$ for $m$ large. By the long exact sequence, it follows that L is $q$-ample. 

It remains to prove the converse of $iii)$. So assume $f$ is finite, surjective and assume $f^*L$ is $q$-ample. We must show that for any coherent sheaf $\F$ on $Y$ we have $H^i(Y,L^{\otimes m} \otimes \F)=0$ for $i>q$ and large $m\gg 0$. We proceed by noetherian induction on $Y$, following the proof of \cite[Proposition 4.4]{Har70}.

By $i)$ and $ii)$, we may reduce to the case where $X$ and $Y$ are projective varieties. In that case, let $d$ be the degree of $f$. If $U=\Spec(A)\subset X$ is an open affine subset, we can choose elements $s_1,\ldots,s_d\in A$ such that $\{s_i\}_{i=1}^d$ is a $K(Y)$-basis of $K(X)$. If $\M$ is the subsheaf of $K(X)$ generated by the $s_i$'s, then by construction the map $\O_Y^d\to f_*\M$ defined by $e_i\mapsto s_i$ is a generic isomorphism. Moreover, applying $\Hom(-,\F)$, we see that this map induces a generic isomorphism $u:f_*\G\to \F^{\oplus m}$ for some coherent sheaf $\G$ on $X$.  Let $K$ and $C$ be the kernel and cokernel of $u$ respectively. By construction, $K$ and $C$ are coherent sheaves supported on a closed proper subset of $Y$. Therefore, by induction, we have $H^i(Y,K\otimes L^{\otimes m})=H^i(Y,C\otimes L^{\otimes m})$ for $i>q$ and $m\gg 0$. Now taking the cohomology of the sequences
$$
0 \to K \to f_*\G \to \mbox{Im } u \to 0
$$
$$
0 \to \mbox{Im } u \to \F^{\oplus d} \to C \to 0
$$and using \eqref{leray}, we see that also $H^i(Y, L^{\otimes m}\otimes \F)=0$ for $i>q$ and $m\gg 0$. 
\end{proof}


\begin{lemma}\label{-Lnotqample}
Let $X$ be a projective scheme of dimension $n$ and $\O(1)$ an ample line bundle on $X$. Then $\O(-1)$ is not $q$-ample for any $q<n$.
\end{lemma}
\begin{proof}The argument follows that of \cite[Theorem 9.1]{Tot10}. If suffices to show that $H^n(X,\O(-m))\neq 0$ for all $m$ sufficiently large. If $\omega_X$ denotes what Hartshorne calls the dualizing sheaf of $X$ \cite{Har77}, we have a canonical isomorphism
$$
\Hom_X(\O_X(-m),\omega_X) \cong H^n(X,\O_X(-m))^*
$$(see \cite[III.7]{Har77}). The coherent sheaf $\omega_X$ is non-zero on $X$, so in particular for all $m$ large, $\Hom(\O_X(-m),\omega_X)=H^0(X,\omega_X(m))\neq 0$.\end{proof}

\section{Ample subschemes}\label{section:defn}

In this section we define the notion of an ample subscheme. We let $X$ be a projective scheme of dimension $n$ over $k$. As mentioned in the introduction, if $Y$ is a closed subscheme of $X$, we consider the exceptional divisor $E$ on the blow-up $X'=\Bl_Y X$ of $X$ with center $Y$. This makes sense because as noted by Hartshorne \cite{Har70}, many positivity properties of $Y$ are reflected in the complement $X\setminus Y\cong X'\setminus E$. Retaining the notations of the previous section, we make the following definition.
\begin{definition}
Let $Y\subset X$ be a closed subscheme of codimension $r$ and let $\pi:X'\to X$ be the blow-up of $X$ with center $Y$. We call $Y$ \emph{ample} in $X$ if the exceptional divisor $E$ is a $(r-1)-$ample divisor on $X'$.
\end{definition}Observe that if $Y$ is a Cartier divisor, then $X'$ is
canonically isomorphic to $X$, so in this case the definition above coincides with the standard notion of ampleness. 

Recall that the blow-up of $Y$ can be identified with the $\Proj$ of the Rees algebra $R(I_Y)=\bigoplus_{m\ge 0}I_Y^k$. In particular, the inclusion $R(I_Y^m)\subset R(I_Y)$ induces an isomorphism of blow-ups $j:\Bl_{I_Y}(X)\to Bl_{I_Y^m}(X)$ such that $j^*\O(E)=\O(mE)$. Hence a subscheme $Y$ is ample in $X$ if and only if $mY$ is, where $mY$ is the scheme defined by the ideal sheaf $I_Y^m$.

\begin{example}
If $X$ is a Gorenstein projective variety, then finite non-empty subsets of $X$ are ample: Indeed, the exceptional divisor $E$ on the blow-up at these points is effective and so it is $(n-1)$-ample by Lemma \ref{smalli}.
\end{example}

\begin{example}Linear subspaces $\PP^k\subset \PP^n$ are
ample, as can be seen by direct calculation using the usual description of the blow-up as a $\PP^{n-k-1}$-bundle. We skip the details of this computation because the result will be an easy consequence of Proposition \ref{c.i}.
\end{example}

Note that $\O_X(E)$ restricts to the line bundle $\O(-1)$ which is negative on the fibers of $\pi$. The next result makes use of this observation to show that ample subschemes are equidimensional. The proof also shows that the number $q=r-1$ is in some sense the best possible for the $q$-ampleness of $E$. 

\begin{proposition}\label{equidim}
Let $Y\subset X$ be an ample subscheme of codimension $r=\dim X-\dim Y$ and let $y\in Y$ be a closed point. Then the fiber of the blow-up $\pi^{-1}(y)$ in $E$ is $(r-1)$-dimensional. In particular, $Y$ is equidimensional.
\end{proposition}

\begin{proof}
Let $Y_0$ be an irreducible component  of $Y$ containing $y$ and let $Z=\pi^{-1}(y)$ be the fiber and let $E_0$ be a component of $E$ dominating $Y_0$. Since $E_0$ is $(n-1)$-dimensional, it follows that $Z$ is at least $(r-1)$-dimensional. 

On the other hand, since $-E$ is $\pi$-ample, the restriction $-E |_Z$ is an ample divisor on the fiber $Z$. Also, since $Y$ is ample, $E |_Z$ is $(r-1)$-ample on the $Z$. By Lemma \ref{-Lnotqample} this implies $\dim Z\le r-1$ and hence $\dim Z=r-1$. 

Now, if $\dim Y_0 < \dim Y$ then the fiber over a closed point in $Y_0$ has dimension $>r$, contradicting the above.
\end{proof}

On the other hand, we will see that ample subschemes can have embedded components (see Section \ref{alg}).

\begin{cor}
Let $X$ be a normal projective variety and let $D$ be a subscheme of codimension one which is ample in $X$. Then $D$ is a Cartier divisor.
\end{cor}

\begin{proof}
The blow-up $\pi:X'\to X$ along $D$ is a birational morphism which is finite by Proposition \ref{equidim}. Since $X$ is normal, $\pi$ is an isomorphism and $D$ is Cartier by the definition of blowing up.\end{proof}

%
%

Recall that a closed subscheme $Y\subset X$ is locally complete
intersection (lci) if its ideal sheaf $\I_Y$ can be locally generated by a regular sequence. When $X$ is smooth this is equivalent to saying that $\I_Y$ can locally be generated by exactly $r$ elements, where $r$ is the codimension of $Y$ in $X$ \cite[p. 105]{Matsumura}. This is of course satisfied if $Y$ is non-singular. By \cite[p. 105]{Har70} the lci condition is equivalent to the two conditions i) $I_Y/I_Y^2$ is a locally free $\O_Y$-module, and
ii) for each $m\ge 0$ the canonical homomorphism $\Sym^m( I_Y/I_Y^2)\to
I_Y^m/I_Y^{m+1}$ is an isomorphism. In particular, the normal sheaf
$N_{Y|X}=(I_Y/I_Y^2)^*$ can be regarded as a vector bundle on $Y$. 

In addition, when $Y$ is locally complete intersection of a smooth projective variety $X$, 
one can check that the blow-up $X'$ is also locally complete intersection. In particular, $X'$ is Gorenstein so the theory from Section \ref{prelim} can be applied.

Even though most subschemes $Y\subset X$ in this paper will be taken to be locally complete intersection, the definition of an ample subscheme makes sense in any projective scheme $X$ and subscheme $Y$. We note that there are ample subschemes that are not locally complete intersection; fat point subschemes in $\mathbb{P}^2$ provide easy examples. However, the next result shows if the ambient scheme is smooth, then all \emph{reduced} ample subschemes are in fact lci.

\begin{proposition}\label{ample=lci}
Let $Y$ be a reduced ample subscheme of a smooth scheme $X$. Then $Y$ is a locally complete intersection.
\end{proposition}

\begin{proof}
A result of Cowsik and Nori \cite{Cowsik} says that a radical, equidimensional ideal $I$ in a regular local ring $(R,\mathfrak{m})$ is a complete intersection provided the Krull dimension of the special fiber algebra $G(I)=\bigoplus_{k\ge0} I^k / \mathfrak{m} I^k$ equals $\mbox{ht}(I)$. In our case, this latter condition follows immediately from Proposition \ref{equidim}.\end{proof}

\section{Ampleness of the normal bundle}\label{section:nb}
Recall that a vector bundle $\E$ is said to be ample if the line bundle $\O(1)$ is ample on $\PP(\E)$, where $\PP(\E)=\Proj(\Sym^* \E)$ is the variety of codimension-1 subspaces of $\E$. Equivalently, in terms of cohomology, $\E$ is ample if for any coherent sheaf $\F$ on $X$, $H^i(X,\Sym^m \E \otimes \F)=0$ for $i>0$ and $m$ sufficiently large (see \cite[III.1]{Har70}). 

In the case $Y$ is locally complete intersection subscheme and $Y$ is ample in $X$, we now show that the normal bundle of $Y$ is an ample vector bundle.  This is a natural requirement because it guarantees that $Y$ intersects every other subvariety non-negatively \cite[Corollary 8.4.3]{Laz04}.

\begin{proposition}\label{cohomE}
Let $\E$ be a vector bundle of rank $r$ on $X$ and let $\pi:\PP(\E^*)\to X$ be the projectivization of $\E^*$. Then for $m\ge 0$,
\begin{equation}\label{cohomE:eq}
H^i(X,\Sym^{m}\E \otimes \F)\cong H^{r+i-1}(\PP(\E^*),\O(-m-r)\otimes \pi^*\left(\F \otimes  \det \E^*\right))
\end{equation}
In particular, $\E$ is an ample vector bundle if and only if the line bundle $\O(-1)$ is $(r-1)$-ample on $\PP(\E^*)$.
\end{proposition}

\begin{proof}Recall that line bundles on $\PP(\E^*)$ are of the form $\O(a)\otimes \pi^*\L$ where $\L$ is a line bundle on $X$. By a well-known formula for higher direct images of line bundles on $\pi:\PP(\E^*)\to X$, we have for $m\ge 0$, $$\pi_*\O(m)= \Sym^m \E^*,\qquad R^{r-1}\pi_*\O(-m-r)= \Sym^m \E \otimes \det \E,$$and all other direct images vanish (see \cite[Appendix A]{Laz04}). Using the projection formula, we get \eqref{cohomE:eq} by the Leray spectral sequence. Here we are implicitly using the condition char$(k)=0$ for the isomorphism $(\Sym^m \E^*)^*\cong\Sym^m \E$. 

To prove the last statement, note that the above formula implies that $\E$ is ample if $\O(-1)$ is $(r-1)$-ample. Conversely, taking $\F=\L \otimes \det \E$ above, we see that  if $\E$ is ample, we have
$$H^i(\PP(\E^*),\pi^*\L(-m))=0$$ for $m\gg 0$, $i\ge r$ and any line bundle $\L$ on $X$. In other words, high multiples of $\O(-1)$ kill cohomology of any line bundle in degrees $\ge r$ and it is therefore $(r-1)$-ample by Lemma \ref{oneO}.
\end{proof}

\begin{remark}\label{Gammaample} This is one of the few places in the paper where we use the characteristic zero assumption. A variant of the above result
can be obtained in positive characteristic by replacing ampleness of $\E$
by Hartshorne's notion of \emph{$\Gamma$-ampleness}, i.e., that
$H^i(X,\Gamma^m(\E)\otimes \F)=0$ for $i>0$ and $m$ large, where
$\Gamma^m(\E)=(\Sym^m \E^*)^*$  (see \cite[III.4]{Har70} for details).
\end{remark}

\begin{cor}\label{NYample}
Let $Y\subset X$ be a locally complete intersection subscheme of
codimension $r$. Then the normal bundle $\N_{Y|X}$ is ample if and
only if $\O_E(E)=\O_E(-1)$ is $(r-1)$-ample on the exceptional divisor
$E$ of the blow-up. In particular, if $Y$ is ample in $X$, then $\N_{Y|X}$ is an ample vector bundle.
\end{cor}
\begin{proof}
Simply recall that when $Y$ is lci, the exceptional divisor $E$ can be identified with the bundle
$\PP(\N_{Y|X}^*)$ and use the previous proposition. For the last
statement, note that if $E$ is $(r-1)-$ample on the blow-up $X'$, then
so is the restriction $\O_{X'}(E)\big|_E=\O_E(-1)$, and hence $\N_{Y|X}$ is an ample vector bundle.
\end{proof}
Of course, the converse of this corollary is false in general since there are
non-ample divisors with ample normal bundle (e.g. the pullback of an ample divisor on a blow-up). We do however have the following result:

\begin{lemma}\label{stabilizes}
Let $Y$ be a locally complete intersection subscheme of codimension $r$ such that $\N_{Y|X}$ is an ample vector bundle. Then for any coherent sheaf $\F$ on $X'$, the maps
$$
H^i(X',\F\otimes \O_{X'}((m-1)E))\to H^i(X',\F\otimes \O_{X'}(mE))
$$are isomorphisms for $i\ge r$ and $m$ sufficiently large.
\end{lemma}

\begin{proof}
As the vector spaces in question are finite dimensional, it suffices to show that the above maps are eventually surjective. 
Consider the exact sequence 
$$
\cdots \to H^i(X',\F((m-1)E)) \to H^i(X',\F(mE)) \to H^i(E, \F(-m)\big|_E) \to \cdots
$$If $N_{Y|X}$ is ample, the previous corollary implies that the groups on the right vanish for $i\ge r$
and $m$ large, giving the desired conclusion.
\end{proof}



We now turn to zero loci of sections of ample vector bundles, which are in many ways the prototypes of ample subvarieties. 

\begin{proposition}\label{c.i}
Let $\E$ be an ample vector bundle on $X$ of rank $r\le n$ and $Y$ be the zero-set of a global section $s\in H^0(X,\E)$. If the codimension of $Y$ is $r$, then $Y$ is ample in $X$.
\end{proposition}

\begin{proof}
Note that the section $s:\O_X\to \E$ induces a surjection $s^*:\E^*\to \I_Y$. Taking symmetric powers, we get a surjection $\Sym \E^*\to \bigoplus_{m\ge 0} \I_Y^m$. Then taking $\Proj$ this gives an embedding of the blow-up of $X$ with center $Y$ into $\PP(\E^*)$,
$$
i:X'=\Bl_Y X \hookrightarrow \PP(\E^*) 
$$under which $i^*\O_{\PP(\E^*)}(1)=\O_{X'}(-E)$. Now since $\E$ is
ample the line bundle $\O_{\PP(\E^*)}(1)$ is ample on $\PP(\E^*)$ by
Proposition \ref{cohomE}. Restricting this line bundle to $X'$, we see
that $E$ is also $(r-1)$-ample and the result follows.
\end{proof}In particular, taking $\E$ to be a direct sum of ample line bundles, we see that any complete intersection subscheme of $X$ is ample.

\section{Cohomology of the complement}\label{complement}

In the classical setting, if $D$ is an effective ample divisor on $X$,
then the complement $U=X\setminus D$ is affine.  By Serre's
characterisation of affineness this is equivalent to the vanishing of
the cohomology groups $H^i(U,\F)$ for all $i>0$ and any coherent sheaf
$\F$ on $X$. Letting $\cd(U)$ denote the cohomological dimension of
$U$, i.e., the smallest integer $r$ such that $H^i(U,\F)=0$ for all
$i>r$, we will generalize this statement by showing that
$\cd(X\setminus Y)=r-1$ for an ample codimension $r$ subscheme
$Y\subset X$. This result has in turn many implications for the geometric properties of $Y$ (cf. Corollary \ref{COR}). 

Note that since any coherent sheaf on the open subset $U$ extends to a coherent sheaf on $X$, it follows that we need only check the condition $H^i(U,\F)=0$ for sheaves of the form $\F=\G\big|_U$ for $\G$ a coherent sheaf  on $X$.

\begin{proposition}\label{qample:cd}
 Let $X$ be a projective scheme. If $U$ is an open subset of $X$ such that $X\setminus U$ is the support of an effective $q$-ample divisor $D$, then $\cd(U)\le q$.
\end{proposition}

\begin{proof}
We will prove the following: For any quasi-coherent sheaf $\F$ on $X$,
\begin{equation}\label{limH}
  H^i(U,\F\big|_U)=\varinjlim H^i(X,\F(mD))
\end{equation}from which the proposition follows immediately. We first
prove \eqref{limH} for $i=0$, i.e., that the restriction of sections,
$\varinjlim H^0(X,\F(mD))\to H^0(U,\F\big|_U)$ isomorphism. This map is injective by \cite[Lemma II.5.3a]{Har77} and surjective because any global section  $s\in H^0(U,\F\big|_U)$ extends to a section in $H^0(X,\F(mD))$ for some $m\ge 0$ by \cite[Lemma II.5.3b]{Har77}. 

To prove \eqref{limH} in general, we use a $\delta$-functor argument. Consider the functor from the category of quasi-coherent sheaves on $X$ to $k$-vector spaces given by $F(\F)=H^0(X,\F\big|_U)$. This is the composition of the functor $\F\to \F\big|_U$, which is exact, and $\G \to H^0(U,\G)$, which is left-exact. From this it follows that $F$ is left-exact and that the derived functors $R^i F$ coincide with $\F \to H^i(U,\F\big|_U)$, by the Grothendieck spectral sequence. 

Similarly, consider the functor $G(\F)=\varinjlim H^0(X,\F(mD))$ which is also a left-exact functor on quasi-coherent sheaves. Since cohomology commutes with direct limits, the derived functors of $G$ coincide with the functors $\F\to \varinjlim H^i(X,\F(mD))$ for $i\ge 0$. Since $F$ and $G$ are left exact their derived functors $R^i F$, $R^i G$ form universal $\delta$-functors. Finally, by the above, we have $R^0 F = R^0 G$ and so they also have the same higher derived functors. This completes the proof of \eqref{limH}. 
\end{proof}

As an application of this result we prove the following version of the
Lefschetz hyperplane theorem, which was proved for integral cohomology by Sommese
\cite[Proposition 1.16]{Som} under the additional assumption that $D$ is semiample.

\begin{cor}[Generalized Lefschetz hyperplane theorem]
  Let $D$ be an effective $q$-ample divisor on a complex projective variety $X$ such that $X-D$ is non-singular. Then $H^i(X,\QQ)\to H^i(D,\QQ)$ is an isomorphism  for $0\le i< n-q-1$ and injective for $i=n-q-1$.
\end{cor}
\begin{proof}
 It suffices to show the statement for cohomology with coefficients in
 $\CC$. Also, by the long exact sequence of relative cohomology it
 suffices to show that $H^i(X,D,\CC)=0$ for $i< n-q$, or equivalently by Lefschetz duality, that $H^i(X\setminus D,\CC)=0$ for $i> n+q$. But this is clear from the spectral sequence of de Rham cohomology $E_1^{st}=H^s(X-D,\Omega^t)\Rightarrow H^{s+t}(X-D,\CC)$, since by the previous proposition, the groups $H^s(X-D,\Omega^t)$ all vanish for $s+t> n+q$. \end{proof}

\noindent Using this result we find the following version of the Lefschetz hyperplane theorem for ample subvarieties. 
\begin{cor}\label{Lefschetz}
Let $X$ be a smooth complex projective variety and let $Y$ be an ample lci subscheme. Then $H^i(X,\QQ)\to H^i(Y,\QQ)$ is an isomorphism for $i<\dim Y$ and injective for $i=\dim Y$.
\end{cor}

\begin{proof}
Let $r$ denote the codimension of $Y$ in $X$ and let $\pi:X'\to X$ be the blow-up of
$X$ along $Y$. Since the exceptional divisor $E$ is $(r-1)$-ample, it
follows from the previous corollary that $H^i(X,Y;\CC)=H^i(X',E;\CC)=0$ for
$i\le \dim Y$. \end{proof}

The following theorem completely describes  ample lci subschemes in terms of the geometric properties studied in Hartshorne's book \cite{Har70}.
\begin{theorem}\label{amplecd}
 Let $Y$ be a locally complete intersection subscheme of codimension $r$ of a smooth projective variety $X$. Then $Y$ is ample in $X$ if and only if $\N_{Y|X}$ is ample and $\cd(X\setminus Y)=r-1$.
\end{theorem}

\begin{proof}
Suppose first that $Y$ is ample in $X$. By Corollary \ref{NYample} the normal bundle $\N_{Y|X}$ is ample so we must show that $\cd(X\setminus Y)=r-1$. As before, let $X'$ be the blow-up of $Y$. We have $X\setminus Y \cong X'\setminus E$, so it suffices to show that $\cd(X'\setminus E')=r-1$. But since $E$ is $(r-1)$-ample on $X'$, it follows from Lemma \ref{qample:cd} that $\cd(X'\setminus E)\le r-1$. Since in any case $\cd(X\setminus Y)\ge r-1$ for any codimension $r$ subscheme of $X$ \cite[p. 99]{Har70}, we get $\cd(X-Y)=r-1$.

Conversely, suppose now that $\N_{Y|X}$ is ample and that the condition $\cd(U)=r-1$ holds. To show that $E$ is $(r-1)$-ample, it suffices to show that for any locally free sheaf $\F$ on $X$ $H^i(X,\F(mE))=0$ for $i \ge r$ and some $m$ sufficiently large. Now, since $\N_{Y|X}$ is ample, Lemma \ref{stabilizes} shows that for all $i\ge r$ and some large $m$, $H^i(X,\F(mE))\cong \varinjlim H^i(X, \F(kE))$. But by \eqref{limH} the last group equals $H^i(X-Y,\F)$, which is zero by assumption.
\end{proof}

\begin{example}The following example shows that the
ampleness assumption on $N_{Y|X}$ is necessary in the above theorem.
It is a variant of Hironaka's construction of a non-ample divisor $Y$ with affine complement (see
\cite{Goodman}). 

Let $X$ be the nodal threefold $z_0z_1=z_2z_3$ in $\PP^{4}$ with a
singularity at $p=$[0,0,0,0,1] and let $\pi:X'\to X$ be one of the two projective
small resolutions of $X$, with exceptional locus a curve $C$ isomorphic to $\PP^1$. Explicit equations for $X'$ are given by
$x_0y_1=x_1y_0$, $x_2y_1=x_3y_0$ in $\PP^4\times \PP^1$.  Let $Y$ be
the divisor given by $x_0+x_1+x_2=0$. From this it is easily verified that $Y$ is smooth
and that $X'\setminus Y\cong X\setminus \{z_0+z_1+z_2=0\}$ is affine. However, $Y$ is
not an ample divisor, since $Y\cdot C=0$. 
\end{example}

The next result summarises the implications of these two conditions to the geometric properties of $Y$.

\begin{cor}[Properties of ample subschemes]\label{COR}
  Let $Y\subseteq X$ be a non-singular ample subscheme of dimension $s\ge 1$. Then $Y$ satisfies the following:
  \begin{enumerate}[(i)]
  \item The normal bundle of $Y$ is ample. 
   \item $Y$ is numerically positive, i.e., $Y\cdot Z > 0$ for all irreducible $(n-s)$-dimensional varieties $Z$. In particular, $Y$ meets every divisor.
  \item The Lefschetz hyperplane theorem for rational cohomology holds on $Y$, i.e., $$H^i(X,\QQ)\to H^i(Y,\QQ)$$ is an isomorphism for $i<s$ and an injection for $i=s$. 
     \item If $\hat{X}$ denotes the completion of $X$ with respect to $Y$, then for any coherent sheaf $\F$ on $X$, $$H^i(X,\F)\to H^i(\hat{X},\F)$$ is an isomorphism for $i<s$ and an injection for $i=s$. 
      \item $Y$ is G3, i.e., $k(X)\cong k(\hat{X})$.
  \end{enumerate}
\end{cor}

\begin{proof}
$i)$ and $iii)$ follow from Corollary \ref{NYample} and Corollary \ref{Lefschetz} respectively. The remaining statements are mainly consequences of $i)$ together with the condition $\cd(X\setminus Y)=\mbox{codim }Y-1$ and proofs can be found in Hartshorne's book (see \cite[III.3.4, IV.1.1. V.2.1 and p. XI]{Har70}.
\end{proof}

Note in particular that ample subvarieties have positive intersection with all other subvarieties of complementary dimension. Thus $i)+ii)$ can be seen as an analogue of one direction of the Nakai-Moishezon criterion for ample divisors. 

\begin{example} Consider a projective variety $Y$ of
dimension $\ge 1$ embedded as the zero-section in the total space $X=\PP(\E^* \oplus \O_Y)$ of a vector bundle $\E$. Then $Y$ is not ample in $X$ since it does not intersect the hyperplane at infinity.
\end{example}


\section{Further properties of ample subschemes}

\subsection{Ampleness in families}
It is well-known that ampleness of divisors is an open condition in algebraic families. Here we show an analogous statement for ample subschemes of higher codimension. 

\begin{theorem}
Let $f:X\to T$ be a flat, projective morphism and let $Y$ be a lci closed subscheme of $X$ such that $f\big |_Y: Y \to T$ is flat and assume that there is a point $t_0\in T$ such that $Y_{t_0}$ is ample in $X_{t_0}$. Then there is an open neighbourhood $U$ of $t_0$ such that for each $t\in U$, $Y_t$ is ample in $X_t$.
\end{theorem}

\begin{proof}
Since $\O_X$ and $\O_Y$ is flat over $T$, so is the ideal sheaf $I_Y$. Using induction on $i$ and the exact sequence
$$
0 \to \I_Y^{i+1}\to \I_Y^i \to \I_Y^i/\I_Y^{i+1}\to 0
$$we see that in fact all the $\I_Y^i$ are flat as $\O_T$-modules (since $\I_Y^i/\I_Y^{i+1}$ is locally free, hence flat over $T$). This means that the Rees algebra of $Y$ is flat over $T$ and hence also the morphism $X'=\Bl_Y X\to T$ is flat. In particular, this implies that $X'|_{\pi^{-1}(Y_t)}=\Bl_{Y_t} X$ and $E\cap \Bl_{Y_t}X_t = E_t$. Now, by \cite{Tot10} $q$-ampleness is an open condition in characteristic zero, so there exists an open set $U\subset T$ containing $t_0$ such that $E_t$ is $(\codim Y-1)$-ample for all $t\in U$. 
\end{proof}

We note that ampleness is however not a \emph{closed} condition in codimension $\ge 2$. For example,in \cite[III.9]{Har77} there is an example of a flat family of smooth twisted cubics in $\mathbb{P}^3$ degenerating into a scheme corresponding to the ideal $(z^2,yz,xz,y^2w-x^2(x+w))$. The latter subscheme of $\PP^3$ is not ample (e.g. since the fiber $\pi^{-1}(y)$ over $y=[0,0,0,1]$ is 2-dimensional, contradicting Proposition \ref{equidim}).


\subsection{Asymptotic cohomology of powers of the ideal sheaf}
The following proposition says that ampleness of a subvariety is
equivalent to the vanishing of the asymptotic cohomology of
powers of the defining ideal.

\begin{proposition}\label{asymptoticideal}
Let $Y$ be a locally complete intersection subscheme of a smooth projective variety $X$. Then $Y$ is ample in $X$ if and only if for all line bundles $\L$ on $X$ and sufficiently large $m$
$$
H^i(X,\I_Y^m \otimes \L)=0
$$for $i\le \dim Y$.
\end{proposition}

\begin{proof}
  Let $X'$ be the blow-up of $X$ with center $Y$ and let $E$ be the
  exceptional divisor. Since $X'$ is lci, hence Gorenstein, we see from Lemmas
  \ref{oneO} and \ref{smalli} that $Y$ is ample if and only if
  for each line bundle $\G$ on $X'$, $H^i(X',\G \otimes \O(-mE))=0$ for
  $m$ large and $i \le n-\codim Y=\dim Y$. In fact, since 
$\pi^*\L(-E)$ is ample for some sufficiently positive line bundle $L$ on $X$, we see that it suffices to consider sheaves $\G$ of the form $\pi^*L$. 
Since $-E$ is $\pi$-ample, we have for $m$ large,
  $\pi_*\O(-mE)= \I_Y^m$, $R^i\pi_*\O(-mE)=0$ for $i>0$, and hence by the 
  Leray spectral sequence, 
$$
H^i(X',\pi^*L \otimes\O(-mE)) \cong H^i(X,L \otimes \I_Y^m)
$$and the conclusion follows.
\end{proof}


\subsection{Intersections}
The next proposition shows that the intersection of two ample subschemes is again ample, provided that it has the expected codimension.

\begin{proposition}
Let $X$ be a smooth projective variety and let $Y_1$ and $Y_2$ be two locally complete intersection ample subschemes of codimensions $d$ and $e$ respectively. If the intersection $Z=Y_1\cap Y_2$ has codimension $r+s$, then it is ample in $X$. 
\end{proposition}

\begin{proof}
 Note that $Z$ is again locally complete intersection. By Theorem \ref{amplecd} we need only verify that i) $N_{Z|X}$ is an ample vector bundle and ii) $\cd(X\setminus Z)\le r+s-1$. The first part is immediate since $N_{Z|X}$ fits into the exact sequence
$$
0 \to N_{Z|Y_1} \to  N_{Z|X} \to  N_{Y_1|X}\big|_Z \to 0
$$where both $N_{Z|Y_1}=N_{Y_2|X}\big|_Z$ and $N_{Y_1|X}\big|_Z$ are ample vector bundles \cite[p. 84]{Har70}. Next, we compute the cohomological dimension of $U=X-Z=U_1\cup U_2$, where $U_i=X\setminus Y_i$ $i=1,2$. By hypothesis we have $\cd(U_1)=r-1$ and $\cd(U_2)=s-1$. Note first that $\cd(U_1\cap U_2) \le \cd(U_1\times U_2)=r+s-2$, since $U_1\cap U_2$ embeds in the diagonal of $U_1\times U_2$. Now the Mayer-Vietoris sequence shows that $H^i(U, \F)=0$ for $i>r+s-1$ for any coherent sheaf $\F$
on $X$. Hence $\cd(U)=r+s-1$ and $Y_1\cap Y_2$ is an ample subscheme of $X$.\end{proof}

\subsection{Transitivity}
A natural way of constructing subschemes with certain positivity properties is by taking a flag $Y_r\subset Y_{r-1} \subset \cdots \subset Y_0=X$ where each $Y_i$ is an ample divisor (or more generally, subscheme) of $Y_{i-1}$. The next result shows that this process does indeed produce an ample subscheme.
%
%
%
%



 \begin{proposition}\label{transitivity}
 Let $Z\subset Y$ be locally complete intersection subschemes of a smooth projective variety $X$ with $Z$ a locally complete intersection in $Y$. 
 If $Z$ is ample in $Y$ and $Y$ ample in $X$, then $Z$ is ample in $X$.


 \end{proposition}

 \begin{proof}
   Again, to show that $Y$ is ample, it suffices to show that the conditions in Theorem \ref{amplecd} are satisfied. First of all, the exact sequence
$$
0 \to N_{Z|Y} \to N_{Z|X} \to  N_{Y|X}\big |_Z \to 0
$$shows that the normal bundle of $Z$ in $X$ is an ample vector bundle. Let $d=\codim(Y,X)$ and $e=\codim(Z,Y)$. To complete the proof it suffices to check that the
complement $U=X-Z$ has cohomological dimension equal to $d+e-1$. Let $V=X\setminus Y$ and consider the local cohomology sequence
$$
\cdots \to H_{Y \cap U}^i(U,\F) \to H^i(U,\F) \to H^i(V,\F) \to \cdots.
$$Since $H^i(V,\F)=0$ for $i>d$ it suffices to show that $H_{Y \cap U}^i(U,\F)=0$ for all $i\ge d+e$. Now, $Y\cap U$ is locally complete intersection in
$U$ and so by \cite[\S 0]{Og}, the local cohomology sheaves $\H_{Y\cap U}^q(\F)$ vanish for all $q>d$. Also, recall that the local cohomology sheaves $\H_{Y\cap U}^q(\F)$ can be written as the direct limit of coherent sheaves of the form $\sheafExt^k(\O_U/I_Y^n,\F)$ which are all supported on the subscheme $Y\cap U=Y-Z$ which has cohomological dimension $e-1$. Since cohomology commutes with taking direct limits on a noetherian topological space \cite[Proposition III.2.9]{Har77}, we see that the terms in the spectral sequence of local cohomology 
$$
E_2^{pq}=H^p(U,\H_{Y\cap U}^q(\F))\Rightarrow H_{Y\cap U}^{p+q}(U,\F)
$$are zero in all degrees $p+q\ge d+e$ . Thus $U$ has cohomological dimension $d+e-1$ and the proof is complete.

 \end{proof}

\subsection{Pullbacks by finite morphisms}

\begin{proposition}\label{basechange}
  Let $f:X'\to X $ be a finite, flat morphism of projective varieties and let $Y$ be a subscheme of $X$. Then $Y'=f^{-1}(Y)$ is ample in $X'$ if and only if $Y$ is ample in $X$.
\end{proposition}

\begin{proof}
This is basically just a formal consequence of the functoriality of blowing up. Note that the ideal sheaf of $Y'$ in $X'$ is given by $f^*I_Y$. There is a canonical surjection of Rees algebras $\bigoplus f^* \I_Y^m\to \bigoplus \I_Y^m$ and which gives a closed embedding $\Bl_{Y'}X' \hookrightarrow  \Bl_{Y}X \times_X X'=\Proj(\bigoplus f^* \I_Y^m).$ This is actually an isomorphism, since $f$ is finite and flat.  Since finite maps are stable under flat base-change, the induced map $\tilde{f}:\Bl_{Y'}X' \to  \Bl_{Y}X$ is finite. Moreover $\tilde{f}$ is surjective because $\dim \Bl_{Y'}X' =\dim \Bl_{Y}X$. Now the exceptional divisor $E'=\tilde{f}^{*}E$ is $(r-1)$-ample if and only if $E$ is $(r-1)$-ample, by Proposition \ref{basechange}, where $r=\codim Y$.
\end{proof}

In particular, note that the proposition applies to any finite surjective morphism of smooth varieties of the same dimension.

\subsection{The fundamental group of an ample subvariety}
From Section 4 we know that the cohomology groups $H^i(Y,\QQ)$ and $H^i(X,\QQ)$
are closely related if $Y$ is an ample subvariety of $X$. In Section 6 we will give examples to
show that the corresponding isomorphisms can fail for integral cohomology and for fundamental groups. 
On the other hand, we do have the following result:

\begin{theorem}
  Let $X$ be a smooth projective variety over $\CC$ and let $Y$ be a smooth ample subvariety of $X$ of dimension $\ge 1$. Then the inclusion induces a surjection of fundamental groups
$$
\pi_1(Y) \to  \pi_1(X) \to 0
$$
\end{theorem}

\begin{proof}
Since $Y$ is ample, the normal bundle of $Y$ is
ample, so at least the image of $\pi_1(Y)$ is of finite
index in $\pi_1(X)$ by results of \cite{NR98}. Take a finite cover $F:X'\to X$
corresponding to this inclusion of subgroups. Consider the preimage $Y'=F^{-1}(Y)$ in $X'$. $Y'$ is ample in $X'$ by Proposition \ref{basechange}, in particular it is connected and so by restricting $F$, we get an induced covering $f:Y'\to Y$. By the construction of the cover $F$ and the general lifting lemma, there 
exists a lifting $l:Y\to Y'$ which is a section of $f$, and so in particular $f_* : \pi(Y')\to \pi(Y)$ must be a surjection. In particular, the map $f$ must be one-to-one and hence so must $F$. This shows that the index of $i_*\pi_1(Y)\subset \pi_1(X)$ was in fact 1, and hence $\pi_1(Y) \to  \pi_1(X)$ is surjective. 
\end{proof}

\begin{remark}For $Y$ ample in $X$ the map $\pi_1(Y)\to \pi_1(X)$ may even have infinite kernel, because there exist smooth projective varieties $Y$ such that $\pi_1(Y)$ is infinite, but has finite abelianization. An explicit example is given by a so-called 'fake projective plane' which is a complex projective surface of general type with the same Betti numbers as $\PP^2$. Such a $Y$ embeds as an ample subvariety of some projective space $X=\PP^n$ by Theorem \ref{Pn-ample} below. On the other hand, it is known that the fundamental group is isomorphic to a torsion free lattice in $PU(2,1)$, so in particular it is infinite (see \cite{Fake}).\end{remark}

\subsection{Integral closure}\label{alg}
Let $Y\subset X$ be a subscheme and let $I\subset \O_X$ be its corresponding sheaf of ideals. When $Y$ is reduced and $X$ is smooth, then ampleness of $Y$ implies that $Y$ has complete intersection singularities. This is of course not true without the reducedness assumption since $I$ defines an ample subscheme if and only if $I^k$ does for $k>0$. In this section we see what effect varying the scheme structure has on ampleness. In particular, we present an example of a non-Cohen Macaulay ample subscheme.

We recall some definitions about ideal sheaves. The integral closure of an ideal sheaf $I\subset \O_X$ is defined as the ideal $\overline{I}=\mu_*\O_{\overline{X}}(-E)$ where $\mu:\overline{X}\to X$ is the normalization of the blow-up of $X$ along $I$. Equivalently, $\overline{I}$ is the ideal consisting of all elements $r\in \O_X$ satisfying some integral equation $r^n+a_1r^{n-1}+\ldots+a_n=0$ with $a_i\in I^i$ \cite[9.6.A]{Laz04}. An ideal sheaf $J\subset I$ is called a reduction of $I$ if $\overline{J}=\overline{I}$. 

\begin{proposition}
Let $Y$ be a subscheme of a projective variety $X$ of codimension $r$ and let $\bar{Y}$ be its integral closure. Then $Y$ is ample in $X$ if and only if $\bar{Y}$ is.
In particular, a subscheme $Y$ is ample if and only if the subscheme $Y'$ associated to a reduction is.
\end{proposition}

\begin{proof}
The blow-ups of $Y$ and $\bar{Y}$ have the same normalization $\overline{X}$ and the exceptional divisors $E$ and ${\bar
E}$ of $\pi $ and ${  \bar{\pi}}$ pull back to the same divisor $F$ on $\overline{X}$. Since the normalization map is finite surjective we conclude from Proposition \ref{basechange} that $E$ is $(r-1)$-ample if and only if ${\bar E}$ is.

The last part is clear since $Y$ and $Y'$ are ample in $X$ if and only their integral closure $\overline{Y}$ is.\end{proof}






\begin{example} The following is an example of a non-Cohen Macaulay ample subscheme. Let $Y\subset \mathbb{P}^3$ be the subscheme associated to the ideal $I=(x_0^2,x_0x_1x_2,x_1^2)$. Then $Y$ is ample, since its integral closure equals $(x_0,x_1)^2$, and $(x_0,x_1)$ is ample in $\mathbb{P}^3$. On the other hand, $Y$ has an embedded point at $[0,0,0,1]$, so in particular it is not Cohen-Macaulay. 
\end{example}

\section{Ample subschemes of projective space}\label{Pn}
The condition $\cd(X\setminus Y)=\codim(Y)-1$ is well understood in the case where $Y$ is a closed subscheme of complex projective space. This allows us to determine the ample subschemes of $\PP^n$ in terms of their topology:

\begin{theorem}\label{Pn-ample} A smooth complex subscheme $Y\subset \PP^n$ is ample if and only if the maps
  \begin{equation}\label{Lmaps}
    H^i(\PP^n,\QQ)\to H^i(Y,\QQ) 
  \end{equation}
are isomorphisms for $0 \le i<\dim Y$. \end{theorem}

\begin{proof}
This follows from a result of Ogus \cite[Theorem 4.4]{Og}, which says that the condition $\cd(\PP^n\setminus Y)=\mbox{codim} Y- 1$ is equivalent to having the above isomorphisms for $0\le i<\dim Y$. When $Y$ is smooth, its normal bundle $N_{Y|\PP^n}$ is ample, since it is a quotient of $T_{\mathbb{\PP}^n}$, which is ample (for $Y$ only lci this might not be the case). Now the result follows from Theorem \ref{amplecd}.
\end{proof}



\begin{example}
Any connected curve in $\PP^n$ is ample. A smooth surface is ample in $\PP^n$ if and only if it is connected and has zero first Betti number.
\end{example}

\begin{example}
By Poincar\'e duality one can replace the maps \eqref{Lmaps} by the corresponding maps in homology, so that ample subvarieties $Y$ have the same rational homology as $\PP^n$ in degrees $0\le i<\dim Y$. The corresponding statement for integral homology is however not true. For example, an Enriques surface $S\subset \PP^5$ is ample by the above proposition, but $H_1(S,\ZZ)=\ZZ/2\ZZ$.  
\end{example}

\begin{example}
Even though two ample subschemes usually have ample intersection, their union might not be ample. For example, let $Y$ be the union of two skew lines in $\PP^3$. Then $Y$ is not ample in our sense, since $Y$ is not connected. In fact, the blow-up of $Y$ can be identified with the projectivized bundle $\pi :\PP(\O(1,0) \oplus \O(0,1))\to \PP^1 \times \PP^1$. Using this description it is easy to check directly that $E$ is not $1-$ample. Note that on the other hand that smooth conics in $\PP^3$ are ample (being complete intersections), so this example shows that ampleness is not stable under rational equivalence.
\end{example}

\begin{example}
If $Y=\PP^1\times \PP^2$ is embedded in $\PP^5$ by the Segre embedding, then the resulting variety is called the Segre cubic threefold. $Y$ is not ample, since $H^2(\PP^1 \times \PP^2,\QQ)=\QQ^2 \neq \QQ = H^2(\PP^5,\QQ)$. Interestingly, in any positive characteristic $p>0$, we have $\cd(\PP^5\setminus Y)=1$, while in characteristic zero,   $\cd(\PP^5\setminus Y)=2$ (see \cite{Har70}). 
\end{example}




\section{Ample curves in homogeneous varieties}
In the previous section we saw that any non-singular curve $Y\subset X=\PP^n$ is ample. In some sense this is not that surprising, since the transitive group action makes $Y$ move in a large family covering $X$, thus in a sense making any subvariety 'ample'. In this section we investigate whether the same remains true when $\PP^n$ is replaced by a different homogeneous variety.

We first recall some definitions. Let $G$ be the group acting transitively on $X$. Fix a point $y\in Y$ and let $G_Y$ be the subgroup of $G$ generated by $\{g \in G\, |\, g\cdot y \in Y\}$. We say that a subset $Y$ \emph{generates} $X$ if $G_Y$ generated $G$. Note that $G_Y$ is independent of the point $y\in Y$.


\begin{proposition}
Let $Y$ be a smooth curve in a projective homogeneous variety $X$. Then $Y$ is ample in $X$ if and only if $Y$ has ample normal bundle.
\end{proposition}

\begin{proof}
To prove $Y$ is ample in $X$ we use the 'Second vanishing theorem' of Hartshorne-Speiser \cite{Spe80}, which states that $\cd(X-Y)\le \dim(X)-2$ is equivalent to the two conditions $i)$ $Y$ is G3 in $X$, i.e., $K(X)=K(\hat{X})$ and $ii)$ $Y$ meets every divisor of $X$. Now if $N_{Y|X}$ is ample, then by a theorem of Hartshorne \cite{Har71b} it follows that $Y$ generates $X$ and $K(\hat{X})$ is a finite  $K(X)$-module. Moreover, since $X$ is projective, homogeneous, \cite[Theorem 4.3]{BS02} and \cite[Theorem 2]{Ballico} now give that $i)$ and $ii)$ are satisfied.. 
\end{proof}

\begin{cor}
Let $Y$ be be a smooth curve in an abelian variety $X$. Then $Y$ is ample if and only if $Y$ generates $X$.
\end{cor}

\begin{proof}
By \cite{Har71b}, the normal bundle of $Y$ is ample if $Y$ generates $X$. Conversely, if $Y$ does not generate $X$, then by \cite[Theorem 4.3]{BS02} there is an irreducible divisor $D\subset X$ such that $D\cap Y=\emptyset$, so $Y$ cannot be ample. 
\end{proof}

\begin{cor}
Let $Y$ be be a smooth curve in a smooth quadric hypersurface $X\subset \PP^{n+1}$ . Then $Y$ is ample if and only $Y$ is not a line.
\end{cor}

\begin{proof}
By a result of Ballico \cite[Theorem 1]{Ballico}, the normal bundle $N_{Y|X}$ is ample if and only if $Y$ is not a line.\end{proof}

\begin{cor}
Let $X$ be the Grassmannian $Gr(r,n)$  and let $Y$ be a non-singular curve in $X$. Then $Y$ is ample in $X$ if and only if $Y$ does not lie in some $Z_3$ where $Z_3$ is the Schubert variety parameterizing linear subspaces $S\subset \mathbb{C}^{n}$ contained in a fixed $V\simeq \mathbb{C}^{n-1}$ and containing a fixed point $p\in V$.
\end{cor}

\begin{proof}
By Theorem 2.1 in Papantonopoulou's paper \cite{Papantonopoulou}, the normal bundle $N_{Y|X}$ is ample if and only if $Y$ does not lie in a $Z_3$.
\end{proof}

\section{On a Kodaira vanishing theorem for $q$-ample line bundles}
In light of the two Lefschetz hyperplane theorems from Section \ref{complement}, it is natural to ask whether these can be derived from some version of the Kodaira vanishing theorem for $q$-ample line bundles, just as in the case of hypersurfaces. 
More precisely, one could wonder if $q$-ample line bundles $L$ satisfy $H^i(X,L\otimes \omega_X)=0$ for $i>q$. For $q=0$ this is the usual statement of the Kodaira vanishing theorem, while for $q=n-1$ this is consequence of Serre duality. In this section we will see that the above is false in general by exhibiting a threefold $X$ with a 1-ample line bundle $L$ such that $H^2(X,L+K)\neq 0$. We refer the reader to the book \cite{FH} for the following facts about flag varieties and representation theory.

\medskip

Let $G=SL_3(\CC)$ and let $B$ be the Borel subgroup of upper triangular matrices with the associated root system $A_2$. 
We let $\alpha_1,\alpha_2,\alpha_3$ denote the set of positive roots; assuming they have unit lengths, they will be $\alpha_1=(1,0),\alpha_2=(\frac12,\frac{\sqrt{3}}{2}),\alpha_3=(-\frac12,\frac{\sqrt{3}}{2})$. 
 
We consider the three-dimensional flag variety $X=G/B$. There is a
natural isomorphism of the weight lattice $\Lambda$ with the Picard
group $\Pic(X)$ given by $\lambda \mapsto L_\lambda = G\times_B
\CC_\lambda.$ Under this correspondence the canonical divisor $K_X$ is
given by $L_{-2\rho}$, where $\rho=\frac12(\alpha_1+\alpha_2+\alpha_3)$. In this setting, the Borel-Weil-Bott theorem gives a complete description of the cohomology groups of a line bundle $L=L_\lambda$: Either $\lambda+\rho$ lies on the boundary of a fundamental chamber of $W$ in $\Lambda$, in which case $H^i(X,L)=0$ for all $i\ge 0$, or $\lambda+\rho$ is in the interior of a Weyl chamber and 
$$H^i(X,L_\lambda)\cong H^0(X,L_{w(\lambda+\rho)-\rho})\qquad \mbox{if } i=\#\{i : (\lambda+\rho,\alpha_i)<0\}$$
and all the other cohomology groups vanish. Here $w$ denotes the unique element of the Weyl group such that  $w(\lambda+\rho)$ lies in the dominant Weyl chamber. This implies that the $q$-ample cones of $X$ are partitioned into Weyl chambers (see Figure \ref{weight}).

\begin{figure}[h!]
  \centering
    \includegraphics[width=0.3\textwidth]{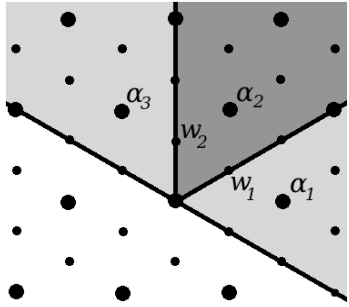}
 \caption{The ample cone (dark grey) and the 1-ample cone (light gray) of $X$}
\label{weight}
\end{figure}

Choose a basis of $\Pic(X)$ corresponding to the two generators
$v_1,v_2$ the dominant chamber; they will correspond to the two divisors generating the nef cone of $X$.
Consider the line bundle given by
$L=L_{(2,-1)}$. On one hand it is easy to check that $L$ is 1-ample by
the Borel-Weil-Bott theorem (since a reflection of it is ample). On the other hand $H^2(X,L+K)=\CC$, since by the theorem, $H^2(X,L+K)=H^2(X,L_{(0,-3)})\cong H^0(X,L_{(0,0)})=H^0(X,\O_X)=\CC$. 

We do however not know of any counterexamples to the above statement with $L$ effective.

\begin{remark}Even though it is well-known that the Kodaira vanishing theorem fails in positive characteristic, the above problem does have a positive answer in characteristic $p$ under some mild assumptions on $X$. This follows from the following vanishing theorem, due to Deligne and Illusie \cite{DI87}:

\begin{theorem}
  Let $X$ be a smooth projective variety of dimension $n$
 over a perfect field  $k$ of characteristic
$p> n$ which admits a lifting to the ring of  Witt vectors $W_2(k)$, and let 
$L$ be a line bundle on $X$. Then for any $m\ge 0$
$$
H^i(X, L^{\otimes p} \otimes \Omega_X^j)=0 \,\forall i+j=m \Longrightarrow \, H^i(X, L \otimes \Omega_X^j )=0\, \forall i+j=m
$$ 
\end{theorem}\noindent In particular, setting $m=n+q+1$ in the above theorem, we see that if $L$ is $q$-ample, then $H^i(X,L \otimes \omega_X)=0$ for all $i>q$.
\end{remark}

\section{Ampleness and the Andreotti-Grauert vanishing theorem}\label{qpositive}
Let $X$ be a smooth complex projective variety of dimension $n$. We call a line bundle $L$ on $X$ \emph{$q$-positive}  if there exists a metric $h$ on $L$ such that the curvature form $\Theta_h(L)$ on $L$ has at least $n-q$ positive eigenvalues at every point. In the 1960s
Andreotti and Grauert \cite{AG} proved that a $q$-positive line bundle is $q$-ample. It is natural to wonder whether the converse is true, i.e., 

\begin{prob}
If $L$ is a $q$-ample line bundle, is $L$ also $q$-positive?
\end{prob}

\noindent This problem was posed by Demailly, Peternell and
Schneider in \cite{DPS96}. For $q=0$, this is a well-known fact in
complex geometry. For the case $q=n-1$, a partial result was proved by Demailly in \cite{Dem10} using holomorphic Morse inequalities. In his recent paper \cite{Mat11}, Matsumura proves the converse for $q=n-1$ under some weak assumptions on $L$ (see \cite[Theorem 1.2]{Mat11} for the precise statement). In particular, his result implies that the converse is true for $X$ a smooth projective surface. He also proves that the converse holds for any $q\ge 0$ under the assumption that $L$ is semiample. In this case, $q$-ampleness also coincides with the notion earlier defined by Sommese in \cite{Som}.

In this section we will show that the answer to the above problem is in general negative for $q$ in the range $\frac{n}2-1<q<n-2$ by constructing explicit counterexamples. For this purpose, we first prove a variant of the Lefschetz hyperplane theorem for $q$-positive line bundles:

\begin{lemma}\label{q-Lefschetz}
If $L$ is a $q$-positive line bundle and let $Z=V(s)\subset X$ be a smooth zero-set of a global section $s\in H^0(X,L)$ of codimension 1. Then the maps
$$
H_i(Z,\ZZ)\to H_i(X,\ZZ) 
$$are isomorphisms for $i<n-q-1$ and surjective for $i=n-q-1$.
\end{lemma}

\begin{proof}
The proof is an elementary application of Morse theory. By assumption,
there is a metric on $L$ such that the curvature form $\Theta=\partial
\overline{\partial} \log |s|^{-2}$ has at least $n-q$ positive
eigenvalues everywhere. Consider the function $\phi(x)=|s(x)|^2$ and
note that $\phi^{-1}(0)=Z$. As in \cite[\S 4]{Bott} one sees that each
component of $Z$ is a critical manifold of $\phi$. By looking at the
Taylor-series of the exponential function, we also see that the
critical points of $\log |s|^2$ and $|s|^2$ on $\phi$ on $X\setminus
Z$ coincide and that their indexes are the same. Since $\Theta$ has
$n-q$ positive eigenvalues at every point, at each critical point of
$\phi$, this means that the real hessian $D^2(\phi)$ is negative definite on a subspace of dimension $\ge n-q$. Using the argument of \cite[Proposition 4.1]{Bott}, we see that it is possible to perturb $\phi$ to have nondegenerate critical points and also all of the above properties and so $\phi$ can be regarded as a Morse function $\phi:X\to [0,\infty)$.

Finally, by the above any critical point of $\phi$ has index at least $n-q$, which by Morse theory means that the homotopy type of $X$ is obtained from $Z$ by attaching cells of dimension at least $n-q$. This completes the proof.
\end{proof}

\begin{lemma}\label{MV}
  Let $\pi:X'\to X$ be the blow-up of a smooth subvariety $S$ of codimension $\ge 2$. Then $H_1(X',\ZZ)=H_1(X,\ZZ)$. 
\end{lemma}

\begin{proof}It is well-known that even the fundamental group $\pi_1(X)$ is a birational invariant (see e.g. \cite{GH}); the essential point is that $\pi_1(X)=\pi_1(X-S)$, since $S$ has real codimension at least four. In particular, $H_1(X',\ZZ)=H_1(X,\ZZ)$ follows by abelianization.
\end{proof}


\subsection{Construction of the counterexamples. } We first give an example of a line bundle on a smooth projective fourfold which is $1$-ample but not $1$-positive. Let $X$ be the blow
up of $\PP^4$ with center a smooth surface $S$ with
$\pi_1(S)=\ZZ/2\ZZ$. Such surfaces were constructed in \cite{Enr}. We
will consider the line bundle $L=\O_X(E)$ where $E$ is the exceptional
divisor on $X$. Here $H^1(S,\QQ)=0$, so $S$ is an ample subvariety by
Theorem \ref{Pn-ample}, and hence the line bundle $L$ is 1-ample
on $X$. There is however no metric on $L$ so that the curvature form
has $3$ positive eigenvalues everywhere. Indeed, if there were, then
by Lemma \ref{q-Lefschetz} we would have
$H_1(X,\ZZ)=H_1(E,\ZZ)=H_1(S,\ZZ)=\mathbb{Z}/2\ZZ$. But this
contradicts Lemma \ref{MV}, since in fact $H_1(X,\ZZ)=H_1(\PP^4,\ZZ)=0$. Hence $L$ is 1-ample, but not $1-$positive.

Thus what makes the counterexample work is precisely the fact that
ample subvarieties do not necessarily satisfy the Lefschetz hyperplane theorems
with integer coefficients. For $n\ge 5$, we can construct similar counterexamples for any $\frac{n}2-1 <q < n-2$ by taking
$S$ instead to be an $s$-dimensional Godeaux-Serre variety, i.e. a quotient of a smooth complete intersection by a free action of a finite group. The variety $S$ embeds into $\mathbb{P}^{n}$ for $n\ge 2s+1$ as an ample subvariety, but the exceptional divisor of the blow-up is not $(n-s-1)$-positive by Lemma \ref{q-Lefschetz}. 

\begin{theorem}
For each $q$ in the range $\frac{n}2-1 <q < n-2$, there exists a line budle on a smooth variety which is $q$-ample, but not $q$-positive.
\end{theorem}

It's worth pointing out that a projective variety $S$ with non-trivial $\pi_1(S)$ cannot be embedded in a projective space of dimension $\le 2s-1$, by the Barth-Larsen theorem \cite[Theorem 3.2.1]{Laz04}, so the above approach does not give counterexamples for $q\le\frac{n}2-1$. So for low $q$ the problem remains open.




\bibliography{ample}

\bibliographystyle{plain}






\end{document}